\documentclass[11pt,twoside,reqno]{amsart}
\usepackage{amsmath}
\usepackage{amsthm}
\usepackage{amsfonts,amssymb}
\usepackage{bbm}
\usepackage{latexsym}
\usepackage{mathrsfs}
\usepackage[all]{xy}
\usepackage{url}
\usepackage[pdftex]{graphicx}
\usepackage{float}
\usepackage{hyperref}
\usepackage{amssymb,yfonts}
\usepackage{fontenc}
 \usepackage{float}
 \usepackage{tabularx}
 
 \usepackage{color,xcolor}

\usepackage{mathdots}

\usepackage[boxsize=23pt]{ytableau}
\numberwithin{equation}{section}

\usepackage[utf8]{inputenc}
\usepackage[english]{babel}
\usepackage{lipsum}
\usepackage{amsthm}
\usepackage{todonotes}
\usepackage[vcentermath]{youngtab}

\theoremstyle{plain}
\newtheorem{lema}{Lemma}[section]
\newtheorem{prop}[lema]{Proposition}
\newtheorem{teo}[lema]{Theorem}

\newtheorem{ques}[lema]{Question}

\theoremstyle{definition}

\newtheorem{ej}[lema]{Example}

\def\ZZ{{\mathbb Z}}
\def\NN{{\mathbb N}}

\def\D{\Delta}

\def\i{\imath}

\def \les {\preceq_{\imath}}
\def \less {\preceq'_{\imath}}

\def \PP{\mathcal{P}_{r}}
\def\cf{\mathfrak{c}}
\def \pf{\mathfrak{p}}
\def \qf{\mathfrak{q}}
\def \ssf{\mathfrak{s}}
\def \cl{\mathrm{cl}}

 \newcommand{\al}{\alpha}

\begin{document}

\title[Complexes of Multichains]{On the homeomorphism and homotopy type of complexes of multichains}

\begin{abstract}
In this paper we define and study for a finite partially ordered set $P$ a class
of simplicial complexes on the set $P_r$ of $r$-element multichains 
from $P$. The simplicial complexes depend on a strictly monotone
function from $[r]$ to $[2r]$. We show that there exactly $2^r$ such functions which yield subdivisions of the order complex of
$P$ of which $2^{r-1}$ are pairwise different. Within this class are for example
the order complexes of the 
interval and the zig-zag poset of $P$ and the $r$\textsuperscript{th} edgewise subdivision of the order complex of $P$. We also exhibit a large subclass for which our simplicial complexes are order complexes and homotopy equivalent to the order complex of $P$.
\end{abstract}



\author[S.~Nazir]{Shaheen Nazir}
\address{Department of Mathematics, Lahore University of Management Sciences, Lahore, Pakistan}
\email{shaheen.nazir@lums.edu.pk}
\author[V.~Welker]{Volkmar Welker}
\address{Philipps-Universit\"at Marburg, Fachbereich Mathematik und Informatik, 35032 Marburg, Germany}
\email {welker@mathematik.uni-marburg.de}

\maketitle
\section{Introduction}

Let $P$ be a poset with order relation $\leq$. In this paper we study the homeomorphism and homotopy type of simplicial complexes associated to multichains in $P$. For a number $r \geq 1$ we consider
the set $P_r$ of all $r$-multichains
$\mathfrak{p} : p_1 \leq \cdots \leq p_r$ in $P$. 

If $r = 1$ then $P_r = P$ and the order complex $\Delta(P)$ of all linearly ordered subsets of $P$ together with its 
geometric realization are well studied geometric and topological objects. They have been shown to encode crucial information about $P$
and have important applications in combinatorics and many other 
fields in mathematics. We refer to reader to \cite{wachs2006poset} for further information and 
background information. 

In this paper we define for $r \geq 2$  classes of simplicial 
complexes associated to $P_r$. In contrast to $r = 1$ there does not
seem to be a canonical choice. In \cite{muehle2015} a poset structure 
on $P_r$ is defined whose order complex can be shown to be homotopy equivalent to $P$. We 
provide this result which was missed in \cite{muehle2015} in 
Proposition \ref{thm:missedbymuehle}. The main focus of our paper is a different construction which leads to a wide class of simplicial complexes associated to $P_r$. Our construction arose during a discussion of preliminary versions of 
work of the first author in \cite{Nazirnew}. As special cases the  construction yields several well
studied and important poset and subdivision operations such as
interval posets (see \cite{walker1988canonical}),
the zig-zag poset (see \cite{peevareinersturmfels}), the $r$\textsuperscript{th}-edgewise subdivision (see \cite{edelsbrunner2000edgewise}) and 
a subdivision operation which arose in the work of Cheeger, M\"uller and Schrader 
(see \cite{CheegerMuller}). We refer the reader to \cite{Nazirnew} for a detailed account of how these arise from our construction (see also Example \ref{ex:general}).

For every strictly monotone map 
$\i : [r]  \rightarrow [2r]$ we define a binary relation
$\preceq_i$ on $P_r$. Here for a natural number $n$ we write $[n]$ for $\{1,\ldots, n\}$. Through the undirected graph 
$G_\i(P_r) = (P_r,E)$ with edge set 
$$E = \big\{ \{\mathfrak{p},\mathfrak{q} \} \subseteq P_r ~:~\mathfrak{p} \preceq_\i \mathfrak{q} \text{ and } \mathfrak{p} \neq \mathfrak{q} \big\}$$
we associate to $P_r$ and $\i$ the clique complex
$\Delta(G_\i(P_r))$ of $G_\i(P_r)$; that is the simplicial complex of all
subsets $A \subseteq P_r$ which form a clique in $G_\i(P_r)$. 

\begin{teo} \label{teo:reflexive}
  For $r \geq 2$ the following are equivalent 
  \begin{itemize}
      \item[(1)] The relation $\preceq_\i$ is reflexive,
      \item[(2)] The clique complex $\Delta(G_\i(P_r))$ is a subdivision of $\Delta(P)$.
      \item[(3)] The clique complex $\Delta(G_\i(P_r))$ is homeomorphic to $\Delta(P)$.
  \end{itemize}
\end{teo}

In the formulation of the theorem and the rest of the paper, we use the term subdivision in 
the sense of geometric subdivision (see e.g. \cite{Stanley}). Also when we speak of
homotopy equivalent or homeomorphic simplicial complexes we mean that their geometric realizations
are homotopy equivalent or homeomorphic.

The cases treated in the theorem include the above mentioned
subdivision operations as special cases.
Indeed, in Lemma \ref{lem:order}(1) we show that for fixed $r$ there are exactly $2^{r-1}$ different
$\i$ for which $\preceq_\i$ is reflexive and $\i(1) = 1$. We will see that the latter condition eliminates an obvious symmetry. 

If $\preceq_\i$ is even a partial order then $\Delta(G_\i(P_r))$ coincides with the order complex
of $P_r$ with respect to $\preceq_\i$. 
The following proposition shows that for each $r$ there is exactly one $\i$ for which $\preceq_\i$ is a partial order.

\begin{prop}
  \label{prop:partial}
  The relation $\les$ is a partial order on $P_r$ if and only if
  for $2 \leq t \leq r$ we have $\imath(t)=\left\{
             \begin{array}{ll}
               2t, & \hbox{$t$ is even;} \\
               2t-1, & \hbox{$t$ is odd.}
             \end{array}
           \right.$
\end{prop}

For the $\i$ from the above proposition and $r=2$ we have that $\preceq_\i$ coincides with
the interval order on $P$ and for
arbitrary $r \geq 2$ we have that
$\preceq_\i$ defines the zig-zag order on $P_r$
(see \cite{peevareinersturmfels} and Example \ref{ex:important}). It can be seen (see \cite{Nazirnew}) that the order complex of the zig-zag order on $P_r$
coincides for even $r$ with a subdivision operation studied in \cite{CheegerMuller}.

The condition that $\preceq_\i$ is reflexive provides the most restrictions on
$\i$ when classifying the $\preceq_\i$ which are partial orders.
Thus instead of the relation $\preceq_\i$, we may consider the relation $\preceq_\i'$ which is
defined as follows:

$$\mathfrak{p} \preceq_\i' \mathfrak{q} :\Leftrightarrow
\left\{ \begin{array}{c} \mathfrak{p} = \mathfrak{q} \text{ or }\\
                      \mathfrak{p} \neq \mathfrak{q} \text{ and }
                      \mathfrak{p} \preceq_\i \mathfrak{q}
                      \end{array} \right.$$

We classify in Proposition \ref{prop:order} when $\preceq_\i'$ is 
partial order and show the following theorem.

\begin{teo} \label{teo:homotopy}
  If $\preceq_\i'$ is a partial order then the order complex of
  $P_r$ with respect to $\preceq_\i'$ is homotopy equivalent to $\Delta(P)$.
\end{teo}

Note that in the situation of 
Theorem \ref{teo:homotopy} we have that the order complex of $P_r$ coincides with 
$\Delta(G_\i(P_r))$ for the $\i$ defining 
$\less$. In case $\les$ is neither reflexive nor transitive then 
there are examples where the conclusion of Theorem \ref{teo:homotopy} fails (see Example \ref{ex:counter}).

The paper is organized as follows. In Section \ref{sec:basic} the 
basic definitions are provided and reflexivity, antisymmetry, transitivity of the relation $\les$ are classified. In particular, Proposition \ref{prop:partial} is proved. 
In Section \ref{sec:top} we prove Theorem \ref{teo:reflexive} and
Theorem \ref{teo:homotopy}. In addition we show in Proposition \ref{prop:different} that there our constructions yields exactly $2^{r-1}$ subdivisions of $\Delta(P)$ yield. Finally, in 
Section \ref{sec:final} we 
outline to relation to the construction from M\"uhle's paper
\cite{muehle2015}, provide an example showing that the conclusion of Theorem \ref{teo:homotopy} is not always true and list a few open questions.

%


\section{The relation $\preceq_i$ on $r$-multichains} \label{sec:basic}

For $1\leq r\in \NN$, let $\imath: [r]\rightarrow [2r] $ be a strictly increasing map. On $P_r$ we define the binary relation $\preceq_\i$ as
follows. For $\mathfrak{p} : p_1 \leq \cdots \leq p_r$ and
$\mathfrak{q} :q_1 \leq \cdots \leq q_r$ we set:

$$\mathfrak{p}\les \mathfrak{q} :\iff
                                         \begin{array}{ll}
                                          p_t\geq q_s , & \hbox{for $s\leq \imath(t)-t$;} \\
                                          p_t\leq q_s, & \hbox{for $s>\imath(t)-t$.}
                                         \end{array}
.
$$ for $\mathfrak{p},\mathfrak{q}\in P_r$.

\begin{ej}
    \label{ex:important}
   \begin{itemize}
       \item For $r =1$ and $\i(1) = 1$ the relation 
       $\preceq_\i$ is the order
       relation $\leq$ on $P$.
       \item If $r = 2$ and $\i(1)=1, \i(2)= 4$ then 
       $\preceq_\i$ is the well studied interval order on $P$ 
       (see \cite{walker1988canonical}); that is the inclusion 
       order on
       the set of intervals $[p_1,p_2] := \{ q\in P~:~p_1 \leq q \leq p_2\}$ for $p_1 \leq p_2$ in $P$.
       \item For arbitrary $r\geq 2$ and $\i(1)=1, \i(2k)=4k,\i(2k+1)=4k+1$,
       $k=1,\ldots,\lfloor\frac{r}{2} \rfloor$, we obtain a partial order
       which appears naturally in the context of monoid algebras and is called zig-zag poset in \cite[p. 390]{peevareinersturmfels} and coincides with the interval order for $r =2$.
   \end{itemize}
\end{ej}

The relation $\preceq_\i$ usually does not even define a partial order.
For example if $r>1$ and $\i(j) = j$ for $j\in [r]$, then 
if $P$ contains a two element chain $p < q$ we have  
for $\mathfrak{p} : \underbrace{p \leq \cdots \leq p}_{r-1} < q$ that
$\mathfrak{p} \not\preceq_\i \mathfrak{p}$ and thus $\preceq_\i$ is not
reflexive.

For  a given strictly increasing map $\i: [r] \rightarrow [2r]$ we write $\{\i_1 <\cdots <\i_r\}$ for the image of $\i$. To $\i$ 
 we associate another strictly increasing map $\i^*:[r]\rightarrow [2r]$ defined as
$\i^*(k)=\i^*_k$, where $\{\i^*_1<\i^*_2<\cdots<\i^*_r\}$ is the complement of  $\{\i_1 <\i_2 <\cdots< \i_r\}$ in $[2r]$. It can be observed that $\les$ is the dual of  $\preceq_{\imath^*}$, i.e. $\pf\les \qf$ if and only if  $\qf \preceq_{\imath^*} \pf$ for all $\mathfrak{p},\mathfrak{q}\in P_r$. Therefore, for our purposes whenever it is convenient, we can restrict ourselves to strictly increasing maps $\i$ with $\i(1)=1$.

Our first goal is to classify when $\preceq_\i$ is a partial order.
The next lemma is a first step towards the classification.

\begin{lema} 
  \label{lem:simple}
  
  \begin{enumerate}
      \item The relation $\les$ is antisymmetric.
      \item The relation $\preceq_i$ is an order for all $\i$ if and only if either $r=1$ or
  $P$ is an antichain.
  \end{enumerate}
\end{lema}
\begin{proof}
   \begin{enumerate}
       \item Let $\mathfrak{p} \preceq_\i \mathfrak{q}$. If
       $\i(j) \leq 2j-1$ then $p_j \leq q_j$ and if
       $\i(j) \geq 2j-1$ then $q_j \leq p_j$. If in 
       addition $\qf\les\pf$ then the fact that $\leq$ is an 
       order relation implies that $p_j=q_j$ for all $j \in [r]$ and hence $\pf=\qf$.
       \item For $r=1$, by $\i(1) =1$ we have $P_1=P$ and $\les=\leq$.
       Thus we may assume $r > 1$.
       
       If $P$ is antichain, then $P_r$ consists of  $r$-multichains of the form 
       $\pf: p\leq\cdots \leq p$ for $p \in P$. From the
       definition of $\preceq_\i$ it follows that $P_r \cong P$ is an antichain as well. 
       
       Conversely, assume $\les$ is an order relation of all $\i$.
       Then the example right before this lemma shows that $P$
       must be an antichain.
   \end{enumerate}
\end{proof}

 Let $\i:[r]\rightarrow [2r]$ with $\i(1)=1$. Then one can write the image of $\i$ as a union of blocks: $\i([r])= B_1\cup \cdots\cup B_\ell$, where each  $B_i$ is an interval of $b_i$ consecutive numbers and $\max B_i+1<\min B_{i+1}$. Setting $b_0=0$, it follows that $B_i=\{\i(b_1+\cdots +b_{i-1}+1),\ldots,\i(b_1+\cdots+ b_i)\}$.
 Analogously, we can write 
 the image of $\i^*$ as a union of blocks $C_1\cup \cdots\cup C_{\ell'}$, where each $C_i$ is an interval of $c_i$ consecutive numbers and $\max C_i+1 < \min C_{i+1}$. We have $\ell' = \ell-1$ if
 $2r$ is in the image of $\i$ and $\ell' = \ell$ otherwise.
 Setting $c_0 = 0$, it follows that 
 $$C_i=\{j\,: \, \i(b_1+\cdots+b_{i})<j<\i(b_1+\cdots+b_i+1)\}=\{\i^*(c_1+\cdots+c_{i-1}+1),\ldots,\i^*( c_1+\cdots+c_{i})\}.$$
 In particular,  we have
 \begin{align*}
   \i(t)  & = t+c_1+\cdots +c_{i-1} \text{ if } t \in B_i \\
   \i^*(t) & = t+b_1+\cdots+b_{i} \text{ if } t\in C_i.
 \end{align*}
For $\ell = \ell'$ the relation $\pf\les \qf$ is then equivalent to the existence of the
following multichain of length $2r$ in $P$ 
          $$\begin{array}{c}
             p_{1}\leq\cdots\leq p_{b_1}\leq  q_{1}\leq\cdots\leq q_{c_1}\leq p_{b_1+1} \leq \cdots \leq p_{b_1+b_2} \leq \cdots\\
\leq  p_{b_1+\cdots+b_{\ell-1}+1}\leq\cdots\leq p_{b_1+\cdots+b_\ell}\leq q_{c_1+\cdots+c_{\ell-1}+1}\leq\cdots\leq q_{c_1+\cdots+c_\ell}
          \end{array}$$

For $\ell' = \ell-1$ then we get the same condition except that the above chain ends at $p_{b_1+\cdots+b_\ell}$.

\begin{lema}
  \label{lem:order}
  Assume $r>1$ and $P$ contains a chain of length one.
  \begin{enumerate}
    \item The relation $\les$ is reflexive if and only if  $\i(t) \in \{2t,2t-1\}$ for all $t\in [r]$.
    \item The relation $\les$ is transitive  if and only if $b_i$ and $c_i$  satisfy the following relations:
  \begin{equation}\label{bc}
    \begin{array}{c}
     b_1<c_1<b_1+b_2<c_1+c_2<\cdots  \\
    < b_1+\cdots+b_{l-1}<c_1+\cdots +c_{\ell-1}\leq b_1+\cdots+b_{\ell}\leq c_1+\cdots +c_{\ell}  
    \end{array}
    \end{equation}
    
  \end{enumerate}
\end{lema}
\begin{proof}
  \begin{enumerate}
    \item  Assume $\preceq_\i$ is reflexive. 
    Let  $t \in [r] \setminus \{1\}$  and let $p<q$ be a two element chain in $P$.  Consider the multichain $\mathfrak{p} : p=p_1\leq \cdots \leq p=p_{t}<q=p_{t+1}\leq \cdots \leq q=p_r$.
    By reflexivity we have $\mathfrak{p} \preceq_\i \mathfrak{p}$.
    Thus $p_{t-1}\leq p_t$ implies that $t-1\leq \i(t)-t$ and  $p_t< p_{t+1}$ implies that $t+1>\i(t)-t$. It follows that $2t-1\leq \i(t)<2t+1$ for all $t\geq 2$.
    
    Conversely, suppose that $\i(1)=1$ and for $t>1$,  $i(t)=2t-1$ or $2t$. Then $b_i\leq 2$ and $c_i\leq 2$ for all $i$. It can easily be verified that $\pf\les \pf$ for all $\pf \in P_r$.
    
    \item 
Assume that \eqref{bc} does not hold.
 Let $1\leq k<\ell$ be the minimal index such that  $b_1+\cdots+b_k\geq  c_1+\cdots+c_k$.  Set
$$\pf: \underbrace{p\leq \cdots\leq p}_{b_1+\cdots+b_k+1}\leq \underbrace{q\leq \cdots \leq q}_{r-b_1-\cdots-b_k-1},~  \qf: \underbrace{p\leq \cdots\leq p}_{c_1+\cdots+c_k}\leq \underbrace{q\leq \cdots \leq q}_{r-c_1-\cdots-c_k},~ \ssf: \underbrace{p\leq \cdots\leq p}_{c_1+\cdots+c_{k-1}}\leq \underbrace{q\leq \cdots \leq q}_{r-c_1-\cdots-c_{k-1}}.$$ Then by definition we have $\pf\les \qf$ and $\qf\les \ssf$ but $\pf\npreceq_{\i}\ssf$.

 Let $1< k<\ell-1$ be the smallest index such that  $c_1+\cdots+c_{k}\geq b_1+\cdots+b_{k+1}$.  Set
$$\pf: \underbrace{p\leq \cdots\leq p}_{b_1+\cdots+b_{k}}\leq \underbrace{q\leq \cdots \leq q}_{r-b_1-\cdots-b_k},~  \qf: \underbrace{p\leq \cdots\leq p}_{b_1+\cdots+b_{k+1}}\leq \underbrace{q\leq \cdots \leq q}_{r-b_1-\cdots-b_{k+1}},~ \ssf: \underbrace{p\leq \cdots\leq p}_{c_1+\cdots+c_{k}+1}\leq \underbrace{q\leq \cdots \leq q}_{r-c_1-\cdots-c_k-1}.$$ Then again by definition we have  $\pf\les \qf$ and $\qf\les \ssf$ but $\pf\npreceq_{\i}\ssf$.

The converse follows from the definition of $\les$ and conditions \eqref{bc}.
  \end{enumerate}
\end{proof}

Proposition \ref{prop:partial} is now an immediate consequence of Lemma \ref{lem:simple} (1) and Lemma \ref{lem:order} (1), (2). Note that in case $\les$ is a partial order then $b_1$=1, $c_1=2$, $c_i=b_i=2$ 
for all $i<\ell$ and $b_\ell=1$, $c_\ell=0$ if $r$ is even and $b_\ell=2$, $c_\ell=1$ if $r$ is odd.

The classification of when $\preceq_\i'$ is a partial order is also an immediate consequence of
Lemma \ref{lem:simple} (1) and Lemma \ref{lem:order} (2).

\begin{prop}
  \label{prop:order}
  The relation $\preceq_\i'$ is a partial order on $P_r$ if and only if
  \eqref{bc} holds.
\end{prop}

%
\begin{ej}
   When $r=2$, there are  only two possibilities for the map $\i$. For $\i(2)=3$, the relation $\les$ is reflexive and  for $\i(2)=4$, it is a partial order.
   In the former case, the clique complex of $G_\i(P_2)$ is the $2$\textsuperscript{nd}-edgewise subdivision of the order complex of $\Delta(P)$. In the latter case, 
   $\preceq_\i$ defines the interval order on $P_r$.  From Walker's result \cite{walker1988canonical} we know that its order complex is homeomorphic to the order complex of $P$.
   
   See Figure \ref{fig 1} for the case $P = \{1 < 2 < 3 \}$ and 
   Example \ref{ex:general} for the extension of both cases to larger $r$.
\end{ej}

\begin{figure}
  \centering
  \includegraphics[width=0.8\textwidth]{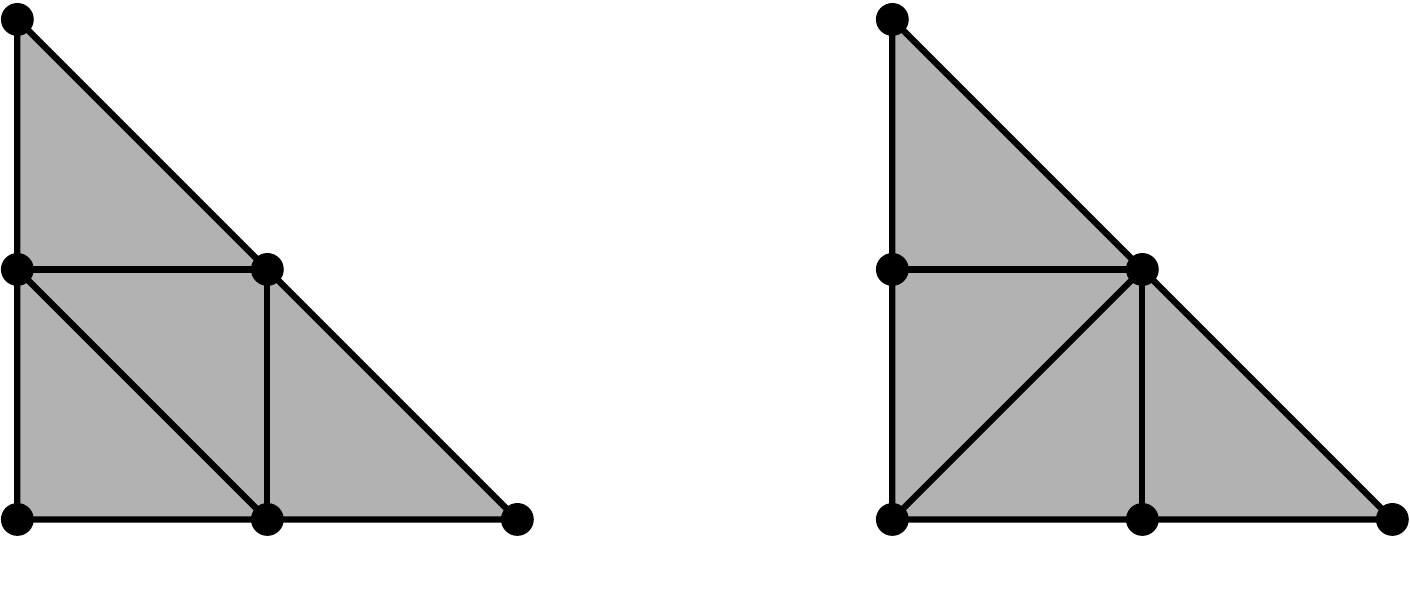}
  \begin{picture}(0,0)(1000,1000)
    \put(670,970){\makebox(0,0)[lb]{(a)~$\i(1) = 1, \i(2) = 3$}}
    \put(890,970){\makebox(0,0)[lb]{(b)~$\i(1) = 1, \i(2) = 4$}}
    \put(920,1000){\makebox(0,0)[lb]{$1\leq 2$}}
    \put(980,1000){\makebox(0,0)[lb]{$1\leq 1$}}
    \put(860,1000){\makebox(0,0)[lb]{$2\leq 2$}}
    \put(705,1000){\makebox(0,0)[lb]{$1\leq 2$}}
    \put(765,1000){\makebox(0,0)[lb]{$1\leq 1$}}
    \put(645,1000){\makebox(0,0)[lb]{$2\leq 2$}}
    \put(625,1140){\makebox(0,0)[lb]{$3\leq 3$}}
    \put(838,1140){\makebox(0,0)[lb]{$3\leq 3$}}
    \put(625,1080){\makebox(0,0)[lb]{$2\leq 3$}}
    \put(838,1080){\makebox(0,0)[lb]{$2\leq 3$}}
    \put(725,1080){\makebox(0,0)[lb]{$1\leq 3$}}
    \put(938,1080){\makebox(0,0)[lb]{$1\leq 3$}}
  \end{picture}%
  \vskip1cm
\caption{$\Delta(G_\i(P_2))$ for $P=\{1<2<3\}$}\label{fig 1}
\end{figure}
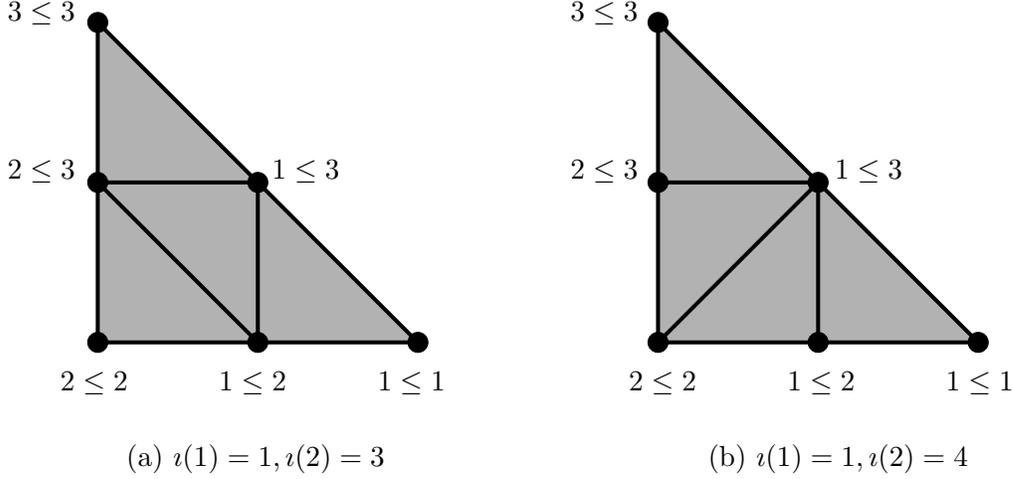
%

\section{Topology of Complexes of Multichains} \label{sec:top}

 In the first part of this section we will prove Theorem \ref{teo:reflexive}. First note that if $P$ is an antichain then by Lemma \ref{lem:simple}(2) $\preceq_\i$ is reflexive for all $\i$. In this situation $P_r$ is an antichain as well and therefore the assertions of Theorem \ref{teo:reflexive} hold.
 
 Hence we may assume that $P$ contains at least a chain of length one. Therefore $\preceq_\i$ is reflexive 
 if and only satisfies the conditions of Lemma \ref{lem:order} (1).
 We first show that in this case 
 $\D(G_{\i}(P_r))$ is a subdivision of $\D(P)$.

\begin{prop}\label{sub}
  Let $P$ be a poset which is not an antichain then $\D(G_{\i}(P_r))$ is a subdivision of $\D(P)$ for every strictly increasing map $\i:[r]\rightarrow [2r]$ for which $\i(1)=1$ and  $\les$ is reflexive.
\end{prop}

In order to prove Proposition \ref{sub} we define the map $f : P_r \rightarrow 
|\D(P)|$ by 
$f(p_1 \leq \cdots \leq p_r)=\frac{1}{r}(p_1+p_2+\cdots+ p_r)$.
Assume $\preceq_\i$ is reflexive. Observe that  if $\pf_1,\ldots, 
\pf_\ell \in P_r$ forms a clique in 
$G_\i(P_r)$ then the union the elements of 
$\pf_1,\ldots,\pf_\ell$ is a chain in 
$P$. Thus we can extend the map
$f$ affinely to a piecewise linear, continuous map $|f|:|\D(G_{\i}(P_r))|\rightarrow |\D(P)|$.
For a chain $\cf : p_1 < \cdots < p_n$ in $P$ let $|\D(G_{\i}(P_r))|_\cf$ be the
set of a $y \in |\D(G_{\i}(P_r))|$ such that if $y = \sum_{i=1}^\ell \mu_i \pf_i$
in strictly positive barycentric coordinates then $\pf_1 \cup \cdots \cup \pf_\ell = \cf$. Building on ideas from \cite[Section 2]{edelsbrunner2000edgewise} we prove the following lemma.

\begin{lema} \label{lem:help}
  For a chain $\cf : p_1 < \cdots < p_n$ in $P$ the restriction of 
  $|f|$ to $|\D(G_{\i}(P_r))|_\cf$ is a homeomorphism to the interior of the simplex spanned by $\cf$ in $|\Delta(P)|$.
\end{lema}
\begin{proof}
   Let $z$ be a point in the interior of the simplex spanned by $\cf$. 
   If we write $z$ in barycentric coordinates as $z=\sum_{i=1}^{n}\lambda_ip_i$ then $\lambda_1,\ldots, \lambda_n > 0$. 
   
   Next we define a matrix 
   $Q_{z,\i} = (q_{i,j}) \in P^{r \times m}$ in the following way. For $1 \leq \ell \leq n$ we set 

  $$\al(\ell) =
   \left\{
  \begin{array}{ll}
    r(\lambda_1+\cdots + \lambda_\ell) -(i-1), & \text{ if \, \,} \genfrac{}{}{0pt}{}{i-1 < r(\lambda_1+\cdots + \lambda_\ell) \leq i }{\text{and } \i(i)=2i-1}\\
    1-\big[ r(\lambda_1+\cdots +\lambda_\ell)
    -(i-1)\big], 
    & \text{ if \,\,} \genfrac{}{}{0pt}{}{i-1 < r(\lambda_1+\cdots + \lambda_\ell) < i}{\text{and } \i(i)=2i}\\
    1,& \text{ if \,\,} \genfrac{}{}{0pt}{}{ r(\lambda_1+\cdots + \lambda_\ell) =i}{\text{and } \i(i)=2i}
  \end{array}
\right. $$  
It is easily seen that $0<\al(\ell)\leq 1$. We denote by
   $(\al_1,\ldots, \al_m)$ the arrangement of the distinct $\al(1),\ldots, \al(n)$ in strictly increasing order.
   Note that by $\alpha(n)= 1$ we will always have $\alpha_m = 1$.
   For numbers $1 \leq i \leq r$ and $1 \leq j \leq m$ we set 
   $$k = \left\{ \begin{array}{ccc} \min\{ k'~:~ r(\lambda_1+\cdots + \lambda_{k'}) \geq i-1+\alpha_{j}\} & \text{ if } \i(i) = 2i-1 \\ 
   \min\{ k'~:~ i-r(\lambda_1+\cdots + \lambda_{k'}) < \alpha_{j}\} & \text{ if } \i(i) = 2i. \end{array} \right. $$
   If we define $q_{i,j} = p_k$ it follows that 
   \begin{align}
     \label{eq:crucial}
        q_{i,1} \leq q_{i,2} \leq \cdots \leq q_{i,n} & \text{ for } \i(i)=2i-1 \\ \nonumber q_{i,m} \leq q_{i,m-1}\leq \cdots \leq q_{r,1} & \text{ for } \i(i)=2i.
    \end{align}
        
    Every element in the $j$\textsuperscript{th} row is smaller than every element in the $(j+1)$\textsuperscript{st} row. 
   In particular, $\qf_j = q_{1,j} \leq \cdots\leq q_{r,j} \in P_r$ for
   $1 \leq j \leq m$. The construction implies that $\qf_j\neq \qf_{j'}$ for $j\neq j'$. Now \eqref{eq:crucial} and
   the definition of $\les$ imply that  
   $$\qf_{1}\prec_{\imath}\qf_{2}\prec_{\i}\cdots \prec_{\i} \qf_{m}.$$
   
Set $x:=\sum_{\ell=1}^{m}(\al_{\ell}-\al_{\ell-1})\qf_{\ell}$ in $|\D(G_{\i}(P_r))|$, where  $\al_{0}=0$. By construction we have that $x \in |\Delta(G_\i(P_r))|_\cf$. From the definition of $|f|$ it follows that 
$|f|(x)=\sum_{\ell=1}^{m}\frac{1}{r}(\al_{\ell}-\al_{\ell-1})(\sum_{i=1}^{n}q_{i,\ell})$.

Let $1\leq k\leq n$. Suppose that  $i_1-1<r(\lambda_1+\cdots+\lambda_{k-1})\leq i_1$ and  $i_2-1<r(\lambda_1+\cdots+\lambda_{k})\leq i_2$ for some $1\leq i_1\leq i_2< r$.

\smallskip

\noindent\textsf{Case 1:} If $\i(i_1)=2i_1-1$ and $\i(i_2)=2i_2-1$, then 
 \begin{equation}\label{r1}
     r\lambda_k=(i_2-1+\al_{t_2})-(i_1-1+\al_{t_1})
 \end{equation}
for some $1\leq t_1,t_2\leq m$.
  
By the definition of  $q_{i,j}$ we have that $p_k$ appears in the matrix $q_{b,\i}$ at positions
$$(i_1,t_1+1),\ldots, (i_2,t_2).$$ 
Thus, in the sum $|f|(x)$, the coefficient of $p_k$ is $\frac{1}{r}\Big((1-\al_{t_1})+(i_2-i_1-1)+(\al_{t_2})\Big)$
which equals $\lambda_k$ by \eqref{r1}.

\smallskip

\noindent\textsf{Case 2:} If $\i(i_1)=2i_1-1$ and $\i(i_2)=2i_2$, then 
 \begin{equation}\label{r2}
r\lambda_k= (i_2-\al_{t_2})-(i_1-1+\al_{t_1})  
\end{equation}
for some $1\leq t_1,t_2\leq m$.
  
Again by the definition of  $q_{i,j}$ we have that $p_k$ appears in the matrix $q_{b,\i}$ at positions
$$(i_1,t_1+1),\ldots, (i_2-1,m), (i_2,t_2+1),\ldots, (i_2,m).$$

Thus, in the sum $|f|(x)$, the coefficient of $p_k$ is $\frac{1}{r}\Big((1-\al_{t_1})+(i_2-i_1-1)+(1-\al_{t_2})\Big)$ which equals $\lambda_k$ by \eqref{r2}.

\smallskip 

\noindent\textsf{Case 3:} If $\i(i_1)=2i_1$ and $\i(i_2)=2i_2-1$, then 
 \begin{equation}\label{r3}
    r\lambda_k = (i_2-1+\al_{t_2})-(i_1-\al_{t_1})
\end{equation}
for some $1\leq t_1,t_2\leq m$.
  
Here $p_k$ appears in the matrix $q_{b,\i}$ at positions
$$(i_1,1),\ldots,(i_1,t_1),(i_1-1,1),\ldots, (i_2,t_2).$$

Thus, in the sum $|f|(x)$, the coefficient of $p_k$ is
$\frac{1}{r}\Big((\al_{t_1})+(i_2-i_1-1)+(\al_{t_2})\Big)$
which again equals $\lambda_k$ by \eqref{r3}. 

\smallskip 

\textsf{Case 4:} If $\i(i_1)=2i_1$ and $\i(i_2)=2i_2$, then 
 \begin{equation}\label{r4}
 r\lambda_k = (i_2-\al_{t_2})-(i_1-\al_{t_1})
\end{equation}
for some $1\leq t_1,t_2\leq m$.
  
 In this case $p_k$ appears in the matrix $q_{b,\i}$ at positions $$(i_1,1),\ldots,(i_1,t_1),(i_1+1,1)\ldots,(i_2-1,m), (i_2,t_2+1),\ldots, (i_2,m).$$
Thus, in the sum $|f|(x)$, the coefficient of $p_k$ is $\frac{1}{r}\Big((\al_{t_1})+(i_2-i_1-1)+(1-\al_{t_2})\Big)$ which again equals $\lambda_k$ by \eqref{r4}. 

\smallskip

The preceding cases imply that  $|f|(x)=z$. In particular, it follows that $|f|$ is surjective. 

For injectivity consider a point $y \in |\Delta(G_\i(P_r))|_\cf$ and
expand $y = \sum_{i=1}^\ell \mu_i \pf_i$ in barycentric coordinates
$\mu_i > 0$ and $\pf_i \in P_r$. Consider $z = |f|(y)$ 
and expands it in terms of barycentric coordinates $z = 
\sum_{i=1}^n \lambda_i p_i$. Going through the steps of constructing
a preimage of $z$ from the proof of surjectivity one checks that 
the preimage coincides with $y$. Now injectivity follows.

As a piecewise linear, continuous and bijective map $|f|$ is a homeomorphism
\end{proof}

\begin{proof}[Proof of Proposition \ref{sub}] By Lemma \ref{lem:help} we know that $|f|$ restricts to a homeomorphism from the interior of a simplex from $|\Delta(P)|$ to the union of the simplices in $|\Delta(G_\i(P_r))|$ whose vertices use all elements of the simplex. Now this fact implies that
$\Delta(G_\i(P_r))$ is a subdivision of 
$\Delta(P)$. 
\end{proof}

Next we show that if $\les$ is not reflexive then $\D(G_\i(P_r))$ is no longer a subdivision of $\D(P)$.
Before we do so we reduce the problem to chains of length $\geq 1$ 
and $\i$ with $\i(1) = 1$ and $\i(2) = 2$.

 \noindent {\sf Reduction to chains:} 
 By Lemma \ref{lem:simple}(2) if $\preceq_\i$ is not reflexive then $P$ must have a chain of length one. If two simplicial complexes $\Delta$ and $\Delta'$ have homeomorphic geometric realizations then their dimension coincide and the subcomplexes generated by the simplices of dimension $\dim(\Delta)$ and $\dim(\Delta')$ respectively have homeomorphic geometric realization 
 as well. Since $\D(G_\i(P_r))$ is the union of $\D(G_\i(C))$ for 
 the maximal chains $C$ in $P$ it therefore suffices to show that 
 for a chain $C$ of length $\geq 1$ 
 either $\dim(\Delta(C)) < \dim \D(G_\i(C))$ 
 or $\D(G_\i(C))$ is not pure. Recall, that a simplicial complex
 is called pure if all its inclusionwise maximal simplices have the same dimension. Hence without loss of generality we may
 assume $P = \{ p_0 < \cdots < p_s \}$ for some $s \geq 1$.
 
 \noindent {\sf Reduction to $\i(1) =1, \i(2) = 2$:} 
 We already have seen that we can assume $\i(1) = 1$. 
 If $\les$ is not reflexive, then $\i$ violates the conditions of Lemma \ref{lem:order}. Let  $k>1$ be the minimal index such that $\i(k)\neq 2k-1,2k$. Then either $\i(k)$ is larger than $2k$ or 
 less than $2k-1$.

 If $\i(k) < 2k-1$ then $\i(k)=2k-2$ and $\i(k-1)=2k-3$ by minimality of $k$. Let $\i_1:[r-k+2]\rightarrow [2(r-k+2)]$ be such that $\i_1(j)=\i(k-2+j)-2(k-2)$. Then $\i_1(1)=1$ and $\i_1(2)=2$. Every $(r-k+2)$-multichain $\pf$ in $P$ can be extended to an $r$-multichain $\pf'$ in $P$ as $\pf'= \underbrace{p_0\leq \cdots \leq p_0}_{k-2}\leq\pf $. Moreover, if $\pf\prec_{\i_1}\qf$ then $\pf'\prec_{\i}\qf'$. It follows that $\Delta(G_{\i_1}(P_{r-k+1}))$
 is a subcomplex of $\Delta(G_\i(P_r))$ of the same dimension.
 
 If $\i(k) > 2k$ then $\i(k) =2k+j$, for some $j\geq 1$. Let $\i_2:[r-k+1]\rightarrow [2(r-k+1)]$ be such that $\i_2(j)=\i(k-1+j)-2(k-1)$. Then $\i_2(1)=j+2$. By duality, we have  $\i^*_2(i)=i$ for all $i\leq j+1$. Every $(r-k+1)$-multichain $\pf$ in $P$ can be extended to a $r$-multichain $\pf'$ in $P$ as $\pf'= \underbrace{p_0\leq \cdots \leq p_0}_{k-1}\leq\pf$. Moreover, if $\pf\succ_{\i^*_2}\qf$ then $\pf'\prec_{\i}\qf'$.
 It follows that $\Delta(G_{\i_1}(P_{r-k+1}))$
 is a subcomplex of $\Delta(G_\i(P_r))$. In particular, if 
 $\dim(\Delta(G_{\i_1}(P_{r-k+1})) > \dim(\Delta(P))$ then
 $\dim(\Delta(G_\i(P_r))) > \dim(\Delta(P)$ and if 
 $\dim(\Delta(G_{\i_1}(P_{r-k+1}))) = \dim(\Delta(P)$ and
 $\Delta(G_{\i_1}(P_{r-k+1}))$ is not pure then so is 
 $\dim(\Delta(G_\i(P_r)))$. Note that it is easily seen that 
 $\dim(\Delta(G_\i(P_r))) \geq \dim(\Delta(P))$ for any $\i$, so the 
 two cases considered in the preceding argument are exhaustive.

 The two reductions imply that it is enough to consider the
 cases covered by the following lemma.

\begin{lema} \label{lem:dim}
  If $P = \{ p_0 < \cdots < p_s\}$ 
  is a chain of length $s \geq 1$ and $\i : [r] \rightarrow [2r]$ such that $\i(1) = 1$ and $\i(2) = 2$
  then either $\dim(\Delta(G_\i(P_r))) > \dim(\Delta(P))$ or $\D(G_{\i}(P_r))$ is not pure.
\end{lema}
\begin{proof}
  Let $b_1,\ldots, b_\ell$ and $c_1,\ldots, c_{\ell'}$ be the numbers
  associated to $\i$ in Section \ref{sec:basic}.
  In the proof we will write $p_0^{a_0}p_1^{a_1}\cdots p_{s}^{a_s}$ for the $r$-multichain
  $$\underbrace{p_0\leq \cdots \leq p_0}_{a_0} \leq
  \underbrace{p_1\leq \cdots \leq p_1}_{a_1} \leq \cdots 
  \leq \underbrace{p_s\leq \cdots \leq p_s}_{a_s},$$
  terms with $a_j = 0$ can be omitted.
  
  \smallskip
  
  \noindent {\sf Case A:} There is no index $t<\ell$ such that $b_1+\cdots+b_t\leq c_1+\cdots+c_t$. 
  
  \smallskip
  
  Thus, 
  if $c_{\ell}\neq 0$, then $c_{\ell}\geq 2$, and if $c_{\ell}=0$, then  $2\leq b_{\ell}\leq c_{\ell-1}$.
 
  \smallskip
  
  \noindent {\sf Subcase A1:}  $b_1\leq c_1$
  
  Then in $G_\i(P_r)$ we have a $(s+2)$-clique of $r$-multichains formed by 
  $$p_0^{b_1}p_s^{r-b_1},\, p_1^{b_1}p_s^{r-b_1}, \ldots ,\, p_{s-1}^{b_1}p_s^{r-b_1},\,p_{s-1}p_s^{r-1},\, p_s^r.$$
  
  It follows that $\dim(\D(G_\i(P_r))) > \dim(\Delta(P))$.

\smallskip

\noindent {\sf Subcase A2:} $b_{\ell}\geq c_{\ell}$, $c_{\ell}\neq 0$

Then in $G_\i(P_r)$ we have a $(s+2)$-clique of $r$-multichains formed by
$$p_0^{r},\, p_0^{r-1}p_1^{r-1},\, p_0^{r-c_{\ell}}p_1^{c_{\ell}},\, p_0^{r-c_{\ell}}p_2^{c_{\ell}}, \ldots, \,p_{1}^{r-c_{\ell}}p_s^{c_{\ell}}.
$$ 

It follows that $\dim(\D(G_\i(P_r))) > \dim(\Delta(P))$.

\smallskip

\noindent {\sf Subcase A3:}  $b_{\ell}\leq c_{\ell-1}$, $c_{\ell}=0$

Then in $G_\i(P_r)$ we have a $(s+2)$-clique of $r$-multichains formed by

$$p_0^{r-b_{\ell}}p_s^{b_{\ell}},\ldots, \, p_0^{r-b_{\ell}}p_2^{b_{\ell}},\, p_0^{r-1}p_1^{r-1},\, p_0^{r}.$$

It follows that $\dim(\D(G_\i(P_r))) > \dim(\Delta(P))$.

\smallskip

\noindent {\sf Subcase A4:} $b_1>c_1$ and $b_{\ell}<c_{\ell}$,
$c_{\ell}\neq 0$

The complex $\D(G_{\i}(P_r))$ has maximal simplices of different dimensions and therefore is not pure. For example
$$p_0^r,\, p_0^{r-1}p_s,\, p_0^{r-c_{\ell}}p_s^{c_{\ell}}$$ and 
$$p_0^{b_1}p_1^{r-b_1},\, p_0p_1^{r-1},\,p_1^r,\, p_1^{r-1}p_2,\,p_1^{r-c_{\ell}}p_2^{c_{\ell}}.$$
form two maximal cliques of different sizes. Maximality can be easily checked.

\smallskip

\noindent {\sf Case B:} 
There is an index $t<\ell$ such that $b_1+\cdots+b_t\leq c_1+\cdots+c_t$. 

\smallskip

Let $t$ be the smallest such index and set $r_0:=b_1+\cdots+b_t$. Let $\i_0: [r_0]\rightarrow [2r_0]$ such that $\i_0(i)=\i(i)$ for all $1\leq i\leq r_0 $. By induction on $r_0$ the claim holds for $\i_0$. 

\smallskip

\noindent {\sf Subcase B1:} 
$\i_0$ satisfies the conditions of one of Subcase A1,A2 or A3

Then in $G_{\i_0}(P_{r_0})$ there is an $(s+2)$-clique of $r_0$-multichains 
$$\pf_0,\,\pf_1,\ldots,\, \pf_{s+1}.$$
This clique gives rise to the $(s+2)$-clique of $r$-multichains in $G_\i(P_r)$ 
$$\pf'_0,\,\pf'_1,\,\ldots, \,\pf'_{s+1},$$
where $\pf'_j=\pf_j\leq \underbrace{p_s\leq \cdots \leq p_s}_{r-r_0}$. Thus, $\dim \D(P)<\dim \D(G_{\i}(P_r))$.

\smallskip

\noindent {\sf Subcase B2:} $\i_0$ satisfies the assumptions of Subcase A4

Then there are two maximal cliques in $G_{i_0}(P_{r_0})$ formed by  
$$p_0^{r_0},\, p_0^{r_0-1}p_s,\, p_0^{r_1-c'_{\ell}}p_s^{c'_{\ell}}$$ and 
$$p_0^{b'_1}p_1^{r_0-b'_1},\, p_0p_1^{r_1-1},\,p_1^{r_0}\,\, p_1^{r_0-1}p_2,\,p_1^{r_0-c'_{\ell}}p_2^{c'_{\ell}}$$  of sizes $3$ and $5$ respectively. 
These cliques give rise to maximal cliques $$p_0^{r_0}p_s^{r-r_0},\, p_0^{r_0-1}p_s^{r-r_0+1},\, p_0^{r_0-c'_{\ell}}p_s^{r-r_0+c'_{\ell}}$$ and 
$$p_0^{b'_1}p_1^{r_0-b'_1}p_s^{r-r_0},\, p_0p_1^{r_0-1}p_s^{r-r_0},\,p_1^{r_0}p_s^{r-r_0},\, p_1^{r_0-1}p_2p_s^{r-r_0},\,p_1^{r_0-c'_{\ell}}p_2^{c'_{\ell}}p_s^{r-r_0}$$ in
$G_\i(P_r)$ of sizes $3$ and $5$ respectively. It follows that $\D(G_\i(P_r))$ is not pure.
\end{proof}

\begin{proof}[Proof of Theorem \ref{teo:reflexive}] From Proposition
\ref{sub} the implication (1) $\Rightarrow$ (2) follows. The implication (2) $\Rightarrow$ (3) is trivial. Now Lemma \ref{lem:dim} says that if $\preceq_\i$ is not
reflexive then either $\dim(\Delta(P)) <
\dim(\Delta(G_\i(P_r)))$ or $\D(G_{\i}(P_r))$ is non-pure. Hence the two cannot have homeomorphic geometric realizations. Therefore (3) implies (1).
\end{proof}

Next we show that for generic $P$ 
there are $2^{r-1}$ different subdivisions
$\Delta(G_\i(P_r))$ for different $\i$.
We collect in  $\mathcal{I}$ the 
 strictly increasing $\i : [r] \rightarrow [2r]$ which satisfy these conditions and in addition $\i(1) = 1$. Since Lemma \ref{lem:order} (1) leaves two choices for $\i(t)$ and $t \in [r] \setminus \{1\}$ it follows that $\mathcal{I}$ consists of
 $2^{r-1}$ maps.

 
 The following lemma is an immediate consequence of the definition of $\les$ and Lemma \ref{lem:order}(1).
 
 \begin{lema}
  Let $\pf:p_1\leq\cdots\leq p_r$, $\qf: q_1\leq \cdots \leq q_r$ be in $P_r$ such that $\pf\neq \qf$. We set
  $K_1:=\{k\in [r]\ :\ p_k<q_k\}$ and $K_2:=\{k\in [r]\  :\ q_k<p_k\}$.
  
  Let $\i\in \mathcal{I}$. Then the following are equivalent:
  \begin{itemize}
    \item[(a)] $\pf\les \qf$ or $\qf \les \pf$
    \item[(b)] $\{ \pf,\qf\}$ is an edge in $G_\i(P_r)$.
    \item[(c)] $\i$ satisfies   $$\i(k)=\left\{
                                                   \begin{array}{ll}
                                                     2k-1 & \text{ for } k\in K_1, \\
                                                     2k & \text{ for } k\in K_2
                                                   \end{array} 
                                                 \right. \text{ and all } k \in [r]
    $$ or $$\i(k)=\left\{
                                                   \begin{array}{ll}
                                                     2k & \text{ for } k\in K_1, \\
                                                     2k-1 & \text{ for } k\in K_2
                                                   \end{array}
                                                 \right.\text{  and all } k \in [r]  
    $$   
    \end{itemize}
\end{lema}

Now we can classify the edges which are present in $G_\i(P_r)$ for all $\i \in \mathcal{I}$.

 \begin{lema}\label{lem: edge}
   Let $\pf:p_1\leq\cdots\leq p_r,\qf: q_1\leq \cdots \leq q_r$ in $P_r$. Then $\{\pf,\qf\}$ is an edge of $G_\i(P_r)$ for each $\i\in \mathcal{I}$ if and only if there is a unique $k \in [r]$ such that $p_k\neq q_k$.
 \end{lema}
 \begin{proof}
 Let $\pf$ and $\qf$ elements of $P_r$ that are connected by an edge in each $G_\i(P_r)$,
$\i \in \mathcal{I}$. Assume that there are two different indices $k_1$ and $k_2$ such that 
$p_{k_1}\neq q_{k_1}$ and $p_{k_2}\neq q_{k_2}$. 

We may assume that either (A) $p_{k_1}<q_{k_1}$ and $p_{k_2}<q_{k_2}$
or (B) $p_{k_1} < q_{k_1}$ and $p_{k_2} > q_{k_2}$.

In situation (A) define $i(k_1)=2k_1-1$ and $\i(k_2)=2k_2$ and extend $\i$ to $[r]$ such that $\i \in \mathcal{I}$. Then neither $\pf\les \qf$ nor $\qf\les\pf$ holds. 

In situation (B) define
$i(k_1)=2k_1-1$ and $\i(k_2)=2k_2-1$ and extend $\i$ to $[r]$ such that $\i \in \mathcal{I}$. Then again  neither $\pf\les \qf$ nor $\qf\les\pf$ holds.

This contradicts the assumption and there is a unique $k$ for which $p_k \neq q_k$.

Conversely,  suppose that $p_k < q_k$
and $p_\ell = q_\ell$ for $\ell \in [r] \setminus \{ k \}$. Then $\pf \les \qf$ for all $\i \in
\mathcal{I}$ with $\i(k)=2k-1$ and $\qf\les\pf$  for all $\i\in \mathcal{I}$ with $\i(k)=2k$. Thus, $\{\pf,\qf\}$ is an edge of $G_\i(P_r)$ for all $\i
\in \mathcal{I}$.
\end{proof}

Now we are position to classify then 
$G_\i(P_r)$ and $G_\jmath(P_r)$ are different for
$\i$ and $\jmath$ from $\mathcal{I}$. 

\begin{prop} \label{prop:different}
  Let $P$ be a poset and $\ell$ the 
  maximal length of a chain in $P$.
  \begin{itemize}
      \item[(1)] If $\ell \leq 1$ then for all $\i,\jmath \in \mathcal{I}$ we have $G_\i(P_r) = G_\jmath(P_r)$. 
      \item[(2)] If $\ell \geq 2$ then $G_{\i}(P_r)\neq G_{\jmath}(P_r)$ for all $\i\neq\jmath\in \mathcal{I}$.
      \end{itemize}
\end{prop}
\begin{proof}
 \begin{itemize}
     \item[(1)] By Lemma \ref{lem: edge} it suffices to show that if $\pf\les \qf$ then either $\pf= \qf$ or there is a unique $k\in[r]$ such that $p_k\neq q_k$.
     
     Suppose that $\pf\neq \qf$ and there are two different indices $k_1 < k_2$ such that $p_{k_1}\neq q_{k_1}$ and $p_{k_2}\neq q_{k_2}$. 
     Since $\ell \leq 1$ and $\pf \les \qf$
     we must have that $\ell = 1$ the elements of
     $\pf$ and $\qf$ come all from a chain $p<q$ of length one in $P$. Then we have either $p_{k_1}=p$, $q_{k_1}=q$, $p_{k_2}=p$ and $q_{k_2}=q$ which implies  $p_{k_2}< q_{k_1}$ or we have $p_{k_1}=q$, $q_{k_1}=p$,  $p_{k_2}=q$, $q_{k_2}=p$ which implies $p_{k_1}>q_{k_2}$. Both cases  contradict the hypothesis that $\pf\les \qf$. 
\item[(2)]
   Let $\{p<q<s\}$ be a chain of length two in $P$. Let $\i\neq\jmath\in \mathcal{I}$. Then there is $k\in [r]$ such that $\i(k)\neq \jmath(k)$. Without loss of generality assume that $\i(k)=2k-1$ and $\jmath(k)=2k$. It is enough to show that there are $\pf\neq \qf$ in $P_r$ such that $\{\pf,\qf\}$ is an edge of $G_{\i}(P_r)$ but $\{\pf,\qf\}$ is not an edge of $G_{\jmath}(P_r)$.  
 Let  $\pf: p<\underbrace{q\leq\cdots\leq q}_{k-1}\leq \underbrace{s\leq \cdots\leq s}_{r-k}$
 and $\qf: \underbrace{q\leq\cdots\leq q}_{k-1}\leq \underbrace{s\leq \cdots\leq s}_{r-k+1}$ in $P_r$. Then $\pf\les\qf$, $\qf\npreceq_{\jmath} \pf$ and 
 $\pf\npreceq_{\jmath} \qf$. Thus, $\{\pf,\qf\}$ is an edge of $G_{\i}(P_r)$  but $\{\pf,\qf\}$ is not an edge of $G_{\jmath}(P_r)$. 
 \end{itemize}
\end{proof}

\begin{ej} \label{ex:general}
   For $\i : [r] \rightarrow [2r]$ with
   $\i(t) = 2t-1$ then $\Delta(G_\i(P_r))$ is the edgewise subdivision of $\Delta(P)$ (see \cite{Nazirnew}).
   
   For $\i : [r] \rightarrow [2r]$ with
   $\i(t) = 2t$ if $t$ is even and $2t-1$ if $t$ is odd then we already know from Proposition \ref{prop:partial} that then 
   $\Delta(G_\i(P_r))$ is an order complex. By the comments after the proposition it is the order complex of the zig-zig order which coincides for even $r$ with a subdivision operation studied in \cite{CheegerMuller} (see \cite{Nazirnew}).
 \end{ej}
 
We have seen that even though postulating $\preceq_\i$ to be reflexive is quite restrictive, the resulting graphs $G_\i(P_r)$ have interesting geometric properties. Next we want to prove Theorem \ref{teo:homotopy}. Recall that the relation $\less$ coincides with $\les$ except that we impose reflexivity.  Since 
by Lemma \ref{lem:simple} (1) $\les$ is always antisymmetric it follows that
if $\les$ is  transitive, then $\less$ becomes a partial order.
Next we will prove Theorem \ref{teo:homotopy} and show that if $\less$ is a partial order then $\Delta(P)$ is homotopy equivalent to the order
complex $\Delta(P_r')$ of $P_r'$.

\begin{proof}[Proof of Theorem \ref{teo:homotopy}]
In the proof we will write $P_r'$ for $P_r$ ordered by  $\less$.

Consider the order complex $\D(P)\setminus\{\emptyset\}$ as a poset. Define a  map  $g:\D(P_r')\setminus\{\emptyset\}
\rightarrow \D(P)\setminus\{\emptyset\}$ as: $$\pf^{(1)}\less\cdots\less \pf^{(k)}\mapsto \cup_{j=1}^{k}\{p_1^{(j)}, \ldots, p_r^{(j)}\}$$ where $\pf^{(j)}: p_1^{(j)}\leq \cdots \leq p_r^{(j)}$. It is clear that $g$ is a poset map. By Quillen's Fiber Lemma, it is enough to show that the fibers $g^{-1}(\D(\{q_0<q_1<\cdots<q_n\})\setminus \{\emptyset\}) = \{ q_0 < q_1 < \cdots < 1_n\}_r'$ of all chains $\{q_0<q_1<\cdots<q_n\}$ in $P$, are contractible. 

 Set $Q=\{q_0 < q_1 <\cdots<q_n\}'_r$.
  Define a map 
  $\mathrm{cl}: Q\rightarrow Q$ by
\begin{eqnarray*}
  \pf:p_1\leq \cdots\leq p_r &\mapsto &\cl(\pf):\underbrace{\underbrace{p_{b_1}\leq \cdots\leq p_{b_1}}_{b_1}\leq \underbrace{p_{b_1+1}\leq \cdots\leq p_{b_1+1}}_{k_1}}_{c_1}  \\
   &\leq & \underbrace{\underbrace{p_{b_1+b_2}\leq \cdots\leq p_{b_1+b_2}}_{m_1}\leq \underbrace{p_{b_1+b_2+1}\leq \cdots\leq p_{b_1+b_2+1}}_{k_2}}_{c_2} \\
   &\leq&\underbrace{p_{b_1+b_2+b_3}\leq \cdots\leq p_{b_1+b_2+b_3}}_{m_2}\leq \cdots,
\end{eqnarray*}
 where  $k_t$ and $m_t$ are the sequences of positive integers such that  $\sum_{i=1}^{t}c_i=\sum_{i=1}^{t}b_i+k_t$ and $\sum_{i=1}^{t+1}b_i=\sum_{i=1}^{t}c_i+m_{t}$ for $t\geq 1$.

\medskip

 \noindent {\sf Claim 1:} 
 The operator $\mathrm{cl}$ is a closure operator on $Q$ (i.e., $\mathrm{cl}$ is a poset map
 such that $\pf \preceq_\i \mathrm{cl}(\pf)$ and $\mathrm{cl} \circ \mathrm{cl} = \mathrm{cl}$)

\medskip

\noindent $\triangleleft$ {\sf Proof of Claim 1:}
First, we show that $\pf\les \cl(\pf)$. Since $b_1<c_1<b_1+b_2<\cdots$, we have
\newline
$\underbrace{p_1\leq \cdots\leq p_{b_1}}_{b_1}\leq \underbrace{p_{b_1}\leq \cdots p_{b_1}\leq p_{b_1+1}\leq \cdots \leq p_{b_1+1}}_{c_1}\leq \underbrace{p_{b_1+1}\leq \cdots \leq p_{b_1+b_2}}_{b_2} \\
\leq \underbrace{p_{b_1+b_2}\leq \cdots\leq p_{b_1+b_2}\leq p_{b_1+b_2+1}\leq \cdots \leq p_{b_1+b_2+1}}_{c_2}\leq \cdots$\\
which shows that $\pf\les \cl(\pf)$.\\
Next, we  verify that $\cl$ is a poset map. Suppose that $\pf\les \pf'$. Then we have
$$\underbrace{p_1\leq \cdots\leq p_{b_1}}_{b_1}\leq \underbrace{p_1'\leq \cdots\leq p_{c_1}'}_{c_1}\leq\underbrace{p_{b_1+1}\leq \cdots\leq  p_{b_1+b_2}}_{b_2}\leq  \underbrace{p_{c_1+1}'\leq \cdots\leq p_{c_1+c_2}'}_{c_2}\leq \cdots.$$ Now again using the relation of $b$'s and $c$'s, we get
\begin{align*} 
  \underbrace{p_{b_1}\leq \cdots\leq  p_{b_1}}_{b_1} & \leq \underbrace{p_{b_1}'\leq \cdots\leq  p_{b_1}'\leq p_{b_1+1}'\leq \cdots \leq p_{b_1+1}'}_{c_1} \\ & \leq\underbrace{p_{b_1+1}\leq \cdots\leq  p_{b_1+1}\leq p_{b_1+b_2}\leq \cdots\leq  p_{b_1+b_2}}_{b_2} \\ & \leq \underbrace{p_{b_1+b_2}'\leq \cdots p_{b_1+b_2}'\leq p_{b_1+b_2+1}'\leq \cdots \leq p_{b_1+b_2+1}'\leq}_{c_2}\cdots \end{align*}
which implies $\cl(\pf)\les \cl(\pf')$.

The fact that $\cl(\pf)=\cl(\cl(\pf))$ follows directly from  the definition of $\cl$. $\triangleright$

\medskip

It is well known (see for example \cite[Exercise 5.2.6(a)]{wachs2006poset}) that since 
$\cl$ is closure operator the order complexes
$\Delta(Q)$ and $\Delta(\cl(Q))$ are homotopy equivalent. 

\medskip
Every element $\pf:p_1\leq p_2\leq \cdots \leq p_r$ of $\cl(Q)$ satisfies the following relation:
\begin{equation}\label{h'}
p_j=\left\{
            \begin{array}{ll}
               
               p_{2t-1}, & \hbox{$c_1+\cdots +c_{t-1}<j\leq b_1+\cdots+b_{t}$;} \\
                p_{2t}, & \hbox{$b_1+\cdots+b_{t}<j\leq c_1+\cdots+c_{t}$.}
             \end{array}
          \right.
\end{equation}
for $1\leq t\leq \ell$. Set $r_0=2\ell-2$ if $c_1+\cdots +c_{\ell-1}= b_1+\cdots+b_{\ell}$;  $r_0=2\ell-1$, otherwise. It is clear that $r_0\leq r$ due to the transitivity of $\i$. Let $\i_0 : [r_0] \rightarrow 
[2r_0]$ be defined by $\i_0 (t) = 2t$ if $t$ is 
even and $2t-1$ if $t$ is odd.
Set $Q_0 = \{ q_0 < q_1 < \cdots < q_n\}_{r_0}$ equipped with the relation $\preceq_{\i_0}$. Then by Proposition \ref{prop:order} $Q_0$ is partially ordered by $\preceq_{\i_0}$.

\medskip

\noindent {\sf Claim 2:}
  The poset $\cl(Q)$ is isomorphic to $Q_0$.

\medskip

\noindent $\triangleleft$ {\sf Proof of Claim 2:}
 Define a map $h:\cl(Q)\rightarrow Q_0$ by:
$$h(\pf)= p_1\leq p_2 \leq \cdots \leq p_{r_0}.$$
If $\pf\less \pf'$ then we get $p_1\leq p_1'\leq p_2'\leq p_2\leq p_3\leq \cdots $ which shows that  $h(\pf)\preceq_{\i_0}h(\pf')$. Thus, $h$ is a poset map.
Now, define $h': Q_0\rightarrow \cl(Q)$ as
$$p_1\leq p_2\leq \cdots \leq p_{r_0} \mapsto h'(p),$$ where $h'(p)$ is defined as in  \eqref{h'} for $b$'s and $c$'s defined by the map $\i$. This is clearly a poset map and an inverse of $h$. Thus $h$ is a poset isomorphism between $\cl(Q)$ and $Q_0$.
$\triangleright$

\medskip

By Theorem \ref{teo:reflexive} we know that $\Delta(Q_0)$ is a subdivision of the $n$-simplex. Hence, $Q$ which is homotopy equivalent to $\cl(Q)\cong Q_0$ is contractible.
\end{proof}

\section{Further directions and open questions} \label{sec:final}

The conclusion of Theorem \ref{teo:homotopy} is false for arbitrary 
$\i$ as the following counterexample shows.

\begin{ej} \label{ex:counter}
  Let $r=3$ and let $P = \{ 1 < 2 < 3\}$ be a three element chain.
  Set $\i(1) = 1$, $\i(2) = 2$ and $\i(3) = 4$. Then $\i$ is
  neither reflexive nor transitive. The maximal simplices of 
  $\Delta(G_\i(P_3))$ are
  $$\begin{array}{ccc}
  \{ 111, 112, 122 \} & \{ 111, 112, 123 \} & \{ 111, 112, 133 \} \\
  \{ 111, 113, 133 \} & \{ 112, 123, 233 \} & \{ 113, 133, 333 \} \\
  \{ 113, 233, 333 \} & \{ 123, 233, 333 \} & \{ 223, 233, 333 \} 
  \end{array}$$
  and
  
  $$\{ 112, 122, 222, 223, 233 \}.$$
  Note that here we write $ijk$ for element $i\leq j\leq k$ from $P_3$.
  
  A straightforward homology computation shows that the reduced homology group
  $\widetilde{H}_i(\Delta(G_\i(P_3)),\ZZ) = 0$ for $i \neq 1$ and $\ZZ$ 
  for $i = 1$. In particular, $\Delta(G_\i(P_3))$ cannot be contractible. Indeed, by a sequence of elementary collapses one 
  can show that $\Delta(G_\i(P_3))$ is homotopy equivalent to a
  circle.
  
  Since $\Delta(P)$ is a $2$-simplex and hence contractible we have 
  that $\Delta(G_\i(P_3))$ and $\Delta(P)$ are not homotopy equivalent.
 \end{ej}

  \begin{ques} \label{que:first} ~
  
    \begin{enumerate}
        \item Can one classify the $\i : [r] \rightarrow [2r]$ for which 
        $\Delta(P)$ and $\Delta(G_\i(P_r))$ are homotopy equivalent 
        for all $P$ ?
        \item Can the homotopy type of $\Delta(G_\i(P_r))$ be described in terms of the homotopy type of $P$ and $\i$ ? 
    \end{enumerate}  
\end{ques}

In \cite{muehle2015} M\"uhle defines a poset structure on 
$P_r$ by setting 

$$\mathfrak{p} \sqsubseteq \mathfrak{q} :\Leftrightarrow 
\, p_i \leq q_i \text{ for } i=1,\ldots, r$$
for multichains $\mathfrak{p} : p_1 \leq \cdots \leq p_r$ and
$\mathfrak{q} : q_1 \leq \cdots \leq q_r$.

It is easily seen that this partial order does not coincide with any of our constructions.
Nevertheless, it exhibits similar topological properties.

\begin{prop}
  \label{thm:missedbymuehle}
  The order complex $\Delta(P_r)$ of $P_r$ with order relation $\sqsubseteq$ is homotopy equivalent to the order complex of $P$.
\end{prop}  
\begin{proof}
 Consider the map
  
  $$\phi : \left\{ \begin{array}{ccc} P_r & \rightarrow & P_{r}
 \\
                  p_1 \leq \cdots \leq p_r & \mapsto &
                  p_r \leq \cdots \leq p_{r} \end{array}
                  \right. .$$

Clearly, $\mathfrak{p} \sqsubseteq \mathfrak{q}$ implies 
$\phi(\mathfrak{p}) \sqsubseteq \phi(\mathfrak{q})$ for $\mathfrak{p},
\mathfrak{q} \in P_r$. 
The map $\phi$ also satisfies $\mathfrak{p} \sqsubseteq \phi( \mathfrak{p})$ and $\phi(\phi(\mathfrak{p})) = \phi(\mathfrak{p})$
and hence is a closure operator. Thus the order complex of
$P_r$ is homotopy equivalent to $\phi(P_r)$ (see \cite[Exercise 5.2.6(a)]{wachs2006poset}). But 
$\phi(P_r)$ is isomorphic to $P$ which implies the assertion.
\end{proof}

Note that the preceding result in not implied by and does not imply
the topological consequences of \cite[Proposition 3.5]{muehle2015}.
In that proposition $P$ has a unique minimal and unique maximal 
element which implies that $P_r$ ordered
by $\sqsubset$ has a unique minimal and a unique
maximal element. The order complex is then formed after
removing these elements. 

One can imagine a common generalization of the relations $\sqsubseteq$ and $\preceq_\i$.
For that consider a strictly monotone map $\i : [r] \rightarrow [2r]$ with associated
blocks $B_1,\ldots ,B_\ell$. For a map  $\kappa : [\ell] \rightarrow \{ \pm 1\}$ define the relation $\preceq_{i,\kappa}$ as follows.
For $\mathfrak{p} : p_1 \leq \cdots \leq p_r$ and
$\mathfrak{q} :q_1 \leq \cdots \leq q_r$ we set:

$$\mathfrak{p}\preceq_{\i,\kappa} \mathfrak{q} :\iff
                                        \left\{  \begin{array}{ll}
                                          p_t\geq q_s , & \text{ for } \i(t) \in B_r, \kappa(r) = 0 \text{ and } s\leq \imath(t)-t, \\
                                          p_t\leq q_s, & \text{ for } \i(t) \in B_r, \kappa(r) = 0 \text{ and } s>\imath(t)-t,\\
                                         p_t\geq q_t , & \text{ for } \i(t) \in B_r, \kappa(r) = 1.
                                          \end{array} \right.
$$ for $\mathfrak{p},\mathfrak{q}\in \PP$.

If $\kappa(r) = 1$ for all $r \in [\ell]$ then $\preceq_{\i,\kappa} = 
\sqsubseteq$ and if $\kappa(r) = 0$ for all $r \in [\ell]$ then $\preceq_{\i,\kappa} = \preceq_\i$.
We can again form a graph $G_{\i,\kappa}(P_r)$ on $P_r$ whose edges are given by $r$-chains $\pf$ and $\qf$ such that either $\pf \prec_{\i,\kappa}
\qf$ or $\qf \prec_{\i,\kappa} \pf$.

The following question partly extends Question \ref{que:first}.

\begin{ques} 
  For which choices of $\kappa$ and 
  $\i$ is $\Delta(G_{\i,\kappa}(P_r))$ homeomorphic (resp.,  homotopy
  equivalent) to $\Delta(P)$ ?
\end{ques}
\bibliographystyle{amsalpha}
\bibliography{References}
\end{document}